\documentclass{amsart}

\title{Non-transversal multilinear duality and joints}

\author{Anthony Carbery and Michael Chi Yung Tang}

\address{School of Mathematics and Maxwell Institute for Mathematical Sciences,
University of Edinburgh, Peter Guthrie Tait Road, Kings Buildings, Edinburgh EH9 3FD}

\email{A.Carbery@ed.ac.uk, Michael.Tang@ed.ac.uk}

\date{2nd January 2022}
\usepackage{amssymb}
\usepackage{amsmath}
\usepackage{physics}
\usepackage{bm}

\usepackage[dvipsnames]{xcolor}

\newtheorem{theorem}{Theorem}

\newtheorem{lemma}{Lemma}
\newtheorem{corollary}{Corollary}

\theoremstyle{remark}

\usepackage{hyperref}
\usepackage{xcolor}
\hypersetup{
  colorlinks   = true, 
  urlcolor     = blue, 
  linkcolor    = blue, 
  citecolor   = red 
}

\usepackage{setspace}
\usepackage{amsmath}
\usepackage{amsfonts}
\usepackage{amssymb}
\usepackage{cases}

\begin{document}

\maketitle


\begin{abstract}
We develop a framework for a duality theory for general multilinear operators which extends that for transversal multilinear operators which has been established in \cite{CHV1}. We apply it to the setting of joints and multijoints, and obtain a ``factorisation" theorem which provides an analogue in the discrete setting of results of Bourgain and Guth (\cite{Guth} and \cite{MR2860188}) from the Euclidean setting. 
\end{abstract}

\section{Introduction}

\noindent
In this note we consider multilinear duality in the context of non-transversality, motivated by the study of joints and multijoints. In what might be called the transversal case, such a theory has been developed in \cite{CHV1} (see also \cite{CHV2} and \cite{CHV3}), and the basic set-up there was as follows.  

\medskip
\noindent
We have a $\sigma$-finite measure space $(X, {\rm d} \mu$), a collection of $d$ normed lattices $\mathcal{Y}_1, \dots , \mathcal{Y}_d$, and for each $j$ a positive linear operator $T_j: \mathcal{Y}_j \to \mathcal{M}(X)$ where $\mathcal{M}(X)$ denotes the space of measurable functions on $X$. Let $\alpha_1, \dots, \alpha_d$ be positive exponents satisfying $\sum_{j=1}^d \alpha_j = 1$. We suppose that we have the ``multilinear" norm bound on the weighted geometric mean $T_1 f_1(x)^{\alpha_1} \dots T_d f_d(x)^{\alpha_d}$ given by
\begin{equation}\label{eq:one}
\| (T_1 f_1)^{\alpha_1} \dots (T_d f_d)^{\alpha_d} \|_{L^q(X)} \leq A  \|f_1\|_{\mathcal{Y}_1}^{\alpha_1} \dots  \|f_d\|_{\mathcal{Y}_d}{^\alpha_d}
\end{equation}
where $ 1 \leq q < \infty$. We also assume that the $T_j$ {\em saturate} $X$, meaning roughly that no subset of $X$ of positive measure remains unreached by all $f_j \in \mathcal{Y}_j$ after the application of $T_j$ (for the formal definition see \cite{CHV1}). The conclusion is that for every nonnegative $M \in L^{q'}(X)$ there are locally integrable functions $g_j$ on $X$ such that 
\begin{equation}\label{eq:two}
M(x) \leq g_1(x)^{\alpha_1} \dots g_d(x)^{\alpha_d} \mbox{ a.e. }
\end{equation}
and for all $j$, for all $f_j \in \mathcal{Y}_j$, we have
\begin{equation}\label{eq:three}
\int_X T_j f_j(x) g_j(x) {\rm d} \mu(x) \leq A \|M\|_{q'} \|f_j\|_{\mathcal{Y}_j}. 
\end{equation}
For reasons set out in \cite{CHV1} this is termed a multilinear duality theorem.
In particular, under the hypotheses described by \eqref{eq:two} and \eqref{eq:three}, \eqref{eq:one} follows by a simple application of H\"older's inequality. The scope of this theorem includes many inequalities arising in multilinear harmonic analysis, especially those which have a transversal geometric set-up. For this reason we refer to this theory as the transversal multilinear duality theory; in the present context this means precisely that we are dealing with a pointwise product of powers of {\em several positive linear operators} $T_j : \mathcal{Y}_j \to \mathcal{M}(X)$. 

\medskip
\noindent
There are, however, many other examples of multilinear geometrical inequalities in harmonic analysis which do not exhibit this transversality property, among them the general endpoint multilinear Kakeya theorem of Bourgain and Guth, \cite{MR2860188}, and the multijoints estimates of Zhang \cite{MR4118611}. These are better modelled in the abstract setting by a {\em single positive multilinear operator} $T$ defined on the product $\mathcal{Y}_1 \times \dots \times \mathcal{Y}_d$. There is no longer any clear role for a collection of possibly different exponents $\alpha_j$, and it seems natural to assign the common value $1/d$ in place of them. In any case, the natural starting point of the non-transversal theory is that of a positive multilinear $T : \mathcal{Y}_1 \times \dots \times \mathcal{Y}_d \to \mathcal{M}(X)$ which saturates $X$ and to assume that we have, for some $q \geq 1$,
\begin{equation}\label{eq:four}
\| T(f_1, \dots, f_d)^{1/d}\|_{L^q(X)} \leq A  \left(\|f_1\|_{\mathcal{Y}_1}\dots  \|f_d\|_{\mathcal{Y}_d}\right)^{1/d}
\end{equation}
in analogy with \eqref{eq:one}.
Can we make conclusions analogous to \eqref{eq:two} and \eqref{eq:three}? It is not perhaps immediately clear what the nature of such conclusions may be, but the following describes one potential set of conclusions which has proved useful in practice.

\medskip
\noindent
{\bf Potential conclusion.} For every $M \in L^{q'}(X)$ of norm $1$, there exist positive linear operators $R_j: \mathcal{Y}_j \to L^1(X)$ such that for every nonnegative  $f_j \in \mathcal{Y}_j$
\begin{equation}\label{eq:five}
M(x)^dT(f_1, \dots, f_d)(x) \leq R_1f_1(x) \dots R_d f_d(x) \mbox{ a.e.,}
\end{equation}
and
\begin{equation}\label{eq:six}
\|R_j\|_{\mathcal{Y}_j \to L^1(X)} \leq A.
\end{equation}
By an even more transparent application of H\"older's inequality, the hypotheses described by \eqref{eq:five} and \eqref{eq:six} readily yield \eqref{eq:four}. Moreover, in the setting where $T(f_1, \dots , f_d)(x)$ happens to be of the form $T_1 f_1(x) \dots T_d f_d(x)$, the potential conclusion coincides with the conclusion described by \eqref{eq:two} and \eqref{eq:three}, where for each fixed $M$ with $\|M\|_{q'} =1$, $R_j$ and $g_j$ are related by
$$ R_j f(x) = g_j(x) T_j f(x).$$
It therefore seems reasonable to hope that under the hypothesis \eqref{eq:four}, we may expect to deduce the conclusion described by \eqref{eq:five} and \eqref{eq:six}.

\medskip
\noindent
Unfortunately this is not the case, as was demonstrated in \cite[Proposition 8.1]{CHV1}, using a concrete  example.

\medskip
\noindent
\subsection{The main result} In this note we demonstrate that in principle one may indeed recover the potential conclusion posited above, under an auxiliary structural hypothesis, at the expense of a larger constant. We then go on to verify the auxiliary hypothesis in a case of current interest in harmonic analysis and discrete geometry, the joints and multijoints estimates of Zhang \cite{MR4118611}. 

\medskip
\noindent
We now state the auxiliary structural hypothesis we impose.

\medskip
\noindent
{\bf Structural hypothesis.} For every nonnegative $(f_j)$ in some dense subspace of $\mathcal{Y}_1 \times \dots \times \mathcal{Y}_d$, there are positive linear operators $S_j : \mathcal{Y}_j \to \mathcal{M}(X)$ such that for all nonnegative $h_j \in \mathcal{Y}_j$,
\begin{equation}\label{eq:seven}
T(h_1, \dots, h_d)(x) \leq S_1h_1(x) \dots S_d h_d(x) \mbox{ a.e.,}
\end{equation}
and 
\begin{equation}\label{eq:eight}
S_1 f_1(x) \dots S_d f_d(x) \leq B^d T(f_1, \dots , f_d)(x) \mbox{ a.e.}\footnote{We thank Timo H\"anninen for pointing out that it is really only an integrated version of this condition that we use.}
\end{equation}
Note that this auxiliary hypothesis is automatically verified in the transversal case with $B = 1$ and $S_j$ indepedent of $(f_j)$. In general, the operators $S_j$ in this structural hypothesis are permitted to depend on the particular inputs $f_j$.

\begin{theorem}\label{thm:main}
Suppose that $(X, {\rm d} \mu)$ is a $\sigma$-finite measure space and $\mathcal{Y}_j$ are normed lattices. Let $T : \mathcal{Y}_1 \times \dots \times \mathcal{Y}_d \to \mathcal{M}(X)$ be a positive multilinear operator which saturates\footnote{i.e. for each subset $E \subseteq X$ of positive measure, there exists a subset $E' \subseteq E$ of positive measure and a nonnegative $(h_1, \dots, h_d) \in Y_1 \times \dots \times Y_d$ such that $T(h_1, \dots , h_d) > 0$ a.e. on $E'$.} $X$. Assume that we have, for some $q \geq 1$,
\begin{equation*}
\| T(f_1, \dots, f_d)^{1/d}\|_{L^q(X)} \leq A  \left(\|f_1\|_{\mathcal{Y}_1}\dots  \|f_d\|_{\mathcal{Y}_d}\right)^{1/d}.
\end{equation*}
Assume moreover that $T$ satisfies the auxiliary structural hypothesis given by \eqref{eq:seven} and \eqref{eq:eight}. Then for every nonnegative $M \in L^{q'}(X)$ with $\|M\|_{q'} =1$ there exist positive linear operators $R_j: \mathcal{Y}_j \to L^1(X)$ such that for every nonnegative  $f_j \in \mathcal{Y}_j$
\begin{equation*}
M(x)^dT(f_1, \dots, f_d)(x) \leq R_1f_1(x) \dots R_d f_d(x) \mbox{ a.e.,}
\end{equation*}
and
\begin{equation*}
\|R_j\|_{\mathcal{Y}_j \to L^1(X)} \leq AB.
\end{equation*}
\end{theorem}
\noindent
In contrast with the transversal theory, Theorem~\ref{thm:main} has content even in the case that the measure space $X$ is a singleton (meaning essentially that we are dealing with multilinear forms rather than multilinear operators). 

\subsection{The symmetric case} If our operator $T$ is symmetric in its arguments (and in particular this requires $\mathcal{Y}_j = \mathcal{Y}$ for all $j$), we can choose the operators $R_j:\mathcal{Y} \to L^1(X)$ in the conclusion of Theorem~\ref{thm:main} to all coincide: if $\rho_j(x,y)$ is the kernel of $R_j$ then we can take the kernel of $R$ to be $\prod_{j=1}^d \rho_j(x,y)^{1/d}$. Moreover, for similar reasons, the auxiliary structural hypothesis in the 
case that $T$ is symmetric is stronger than the following symmetric version:

\medskip
\noindent
{\bf Structural hypothesis (symmetric version).} For every nonnegative $f$ in some dense subspace of $ \mathcal{Y}$, there is a positive linear operator $S : \mathcal{Y} \to \mathcal{M}(X)$ such that for all nonnegative $h_j \in \mathcal{Y}$,
\begin{equation}\label{eq:seven_sym}
T(h_1, \dots, h_d)(x) \leq Sh_1(x) \dots S h_d(x) \mbox{ a.e.,}
\end{equation}
and 
\begin{equation}\label{eq:eight_sym}
Sf(x)^d \leq B^d T(f, \dots , f)(x) \mbox{ a.e.}
\end{equation}
Under this weaker structural hypothesis, and under only the {\em diagonal} version of the main hypothesis \eqref{eq:four}, we can still obtain the conclusion of Theorem~\ref{thm:main} with all the $R_j$ coincident:

\begin{theorem}\label{thm:main_sym}
Suppose that $(X, {\rm d} \mu)$ is a $\sigma$-finite measure space and $\mathcal{Y}$ is a normed lattice. Let $T : \mathcal{Y}^d\to \mathcal{M}(X)$ be a symmetric positive multilinear operator which saturates $X$. Assume that we have, for some $q \geq 1$,
\begin{equation*}
\| T(f, \dots, f)^{1/d}\|_{L^q(X)} \leq A  \|f\|_{\mathcal{Y}}.
\end{equation*}
Assume moreover that $T$ satisfies the auxiliary structural hypothesis given by \eqref{eq:seven_sym} and \eqref{eq:eight_sym}. Then for every nonnegative $M \in L^{q'}(X)$ with $\|M\|_{q'} =1$ there exists a positive linear operator $R : \mathcal{Y} \to L^1(X)$ such that for all nonnegative  $f_j \in \mathcal{Y}$
\begin{equation*}
M(x)^dT(f_1, \dots, f_d)(x) \leq Rf_1(x) \dots Rf_d(x) \mbox{ a.e.,}
\end{equation*}
and
\begin{equation*}
\|R\|_{\mathcal{Y} \to L^1(X)} \leq AB.
\end{equation*}
\end{theorem}
\noindent
It is not clear whether Theorem~\ref{thm:main} also implies Theorem~\ref{thm:main_sym}, even if we are willing to lose constants depending on the degree of multilinearity $d$. While the diagonal condition 
\begin{equation*}
\| T(f, \dots, f)^{1/d}\|_{L^q(X)} \leq A  \|f\|_{\mathcal{Y}}
\end{equation*}
readily implies the off-diagonal condition\footnote{Indeed, assume that for all $f$ we have
\begin{equation*}
\| T(f, \dots, f)^{1/d}\|_{L^q(X)} \leq A  \|f\|_{\mathcal{Y}}.
\end{equation*}
Given $(f_j)$, let $\lambda_j>0$ be such that $\prod_j \lambda_j =1$ and consider $f := \sum_j \lambda_j f_j$. Then
$$\| T(f_1, \dots, f_d)^{1/d}\|_{L^q(X)} = \| T(\lambda_1f_1, \dots, \lambda_d f_d)^{1/d}\|_{L^q(X)} \leq \| T(f, \dots, f)^{1/d}\|_{L^q(X)}$$
$$ \leq A \|f\|_\mathcal{Y} \leq A \sum_j \lambda_j \|f_j\|_\mathcal{Y}.$$
Therefore, by the arithmetic-geometric mean inequality,
$$\| T(f_1, \dots, f_d)^{1/d}\|_{L^q(X)} \leq  A \inf_{\prod_j \lambda_j =1}\sum_j \lambda_j \|f_j\|_\mathcal{Y} =
dA  \left(\|f_1\|_{\mathcal{Y}}\dots  \|f_d\|_{\mathcal{Y}}\right)^{1/d}.$$
At least when $d=2$ this numerology is sharp. Consider the discrete setting in which $X$ is a singleton and the bilinear form $T$ is given by a matrix. If $T$ has constant $1$ in the off-diagonal case, then some entry $(i,j)$ of the matrix $A$ corresponding to $T$ is $1$. Let $\tilde{T}$ be the bilinear form with matrix with entries $1$ in the $(i,j)$ and $(j,i)$ positions, and zero entries elsewhere. Then the diagonal constant for $T$ will be at least as large as it is for $\tilde{T}$, and a direct calculation shows that for $\tilde{T}$ it is exactly $1/2$ when $i \neq j$ and $1$ when $i=j$.}
\begin{equation*}
\| T(f_1, \dots, f_d)^{1/d}\|_{L^q(X)} \leq dA  \left(\|f_1\|_{\mathcal{Y}_1}\dots  \|f_d\|_{\mathcal{Y}_d}\right)^{1/d},
\end{equation*}
it is less clear that the diagonal auxiliary conditions \eqref{eq:seven_sym} and \eqref{eq:eight_sym} imply their off-diagonal counterparts \eqref{eq:seven} and \eqref{eq:eight}.

\medskip
\noindent
It is also not clear to what extent the structural conditions might be necessary in order for the conclusions of  Theorems~\ref{thm:main} and \ref{thm:main_sym} to hold.



\medskip
\noindent
We prove Theorems~\ref{thm:main} and \ref{thm:main_sym} in Section~\ref{sect:two} below in the special case of finite discrete measure spaces $X$,
$Y$ and $Y_j$ (over which the lattices $\mathcal{Y}$ and $\mathcal{Y}_j$ are defined). The details of the arguments for the general cases will appear elsewhere.

\medskip
\noindent
\subsection{Joints and multijoints}
Joints and multijoints problems can be regarded as discrete analogues of the Kakeya and multilinear Kakeya problems on Euclidean spaces. Let $\mathbb{F}$ be a field and let $\mathcal{L}$ be a family of lines in $\mathbb{F}^d$. 
A {\bf joint} of $\mathcal{L}$ is a meeting point in $\mathbb{F}^d$ of $d$ lines in $\mathcal{L}$ which have linearly independent directions. If we have $d$ families of lines $\mathcal{L}_1, \dots , \mathcal{L}_d$ in $\mathbb{F}^d$, a {\bf multijoint} is a joint for $\mathcal{L}_1 \cup \dots \cup \mathcal{L}_d$
{\em with the additional restriction that exactly one line forming the joint comes from each family $\mathcal{L}_j$.} We denote by $J$ the set of joints or multijoints formed by a family or families of lines, according to context.

\medskip
\noindent
For the joints problem we define 
$$ N(x) : = \#\{(l_1, \dots , l_d) \in \mathcal{L}\, : \, l_1, \dots, l_d \mbox{ form a joint at }x\}$$
and for the multijoints problem
$$ N(x) : = \#\{(l_1, \dots , l_d) \in \mathcal{L}_1 \times \dots \times \mathcal{L}_d\, : \, l_1, \dots, l_d \mbox{ form a joint at }x\}.$$
We allow repetitions in the families $\mathcal{L}$ and $\mathcal{L}_j$, and our definition of $N(x)$, as well as the cardinalities $|\mathcal{L}|$ and $|\mathcal{L}_j|$, are understood to count such repetitions.

\medskip
\noindent
R.~Zhang \cite{MR4118611} has established the following sharp estimates:

\begin{theorem}\label{thm:zhang}
For the joints problem we have
$$ \sum_{x \in J} N(x)^{1/(d-1)} \lesssim  |\mathcal{L}|^{d/(d-1)}$$
and for the multijoints problem we have
$$ \sum_{x \in J} N(x)^{1/(d-1)} \lesssim \left(|\mathcal{L}_1| \dots |\mathcal{L}_d|\right)^{1/(d-1)}$$
where the implicit constants depend only on $d$.
\end{theorem}

\noindent
We can take advantage of the possibility of repetitions, together with the scaling enjoyed by the estimates in this result, and also the density of the rationals in the real numbers, to see that the joints and multijoints problems fit into the framework
we have set out above. Let $\mathcal{Y}_j = l^1(\mathcal{L}_j)$ with counting measure (and we now assume the lines in $\mathcal{L}_j$ are distinct), and let $X$ be the set of multijoints of $\mathcal{L}_1, \dots , \mathcal{L}_d$, again with counting measure. For $f_j \in l^1(\mathcal{L}_j)$ let
$$ T(f_1, \dots , f_d)(x) 
= \sum_{l_j \in \mathcal{L}_j} \delta(x,l_1, \dots , l_d) f_1(l_1) \dots f_d(l_d)$$
where the {\bf multijoints kernel} $\delta$ is given by $\delta(x,l_1, \dots, l_d) = 1$ if  $x \in l_1, \dots ,l_d$ and the directions of $l_1, \dots , l_d$ are linearly independent, and by $\delta(x,l_1, \dots, l_d) =0$ otherwise. The second estimate of Theorem~\ref{thm:zhang} then gives 
$$\|T(f_1, \dots , f_d)^{1/d}\|_{d/(d-1)} \lesssim \left(\|f_1\|_1 \dots \|f_d\|_1\right)^{1/d}.$$
Similarly, if we take all the families $\mathcal{L}_j$ to be a common family $\mathcal{L}$, the first estimate of Theorem~\ref{thm:zhang} then gives the symmetric-form inequality
$$\|T(f, \dots , f)^{1/d}\|_{d/(d-1)} \lesssim \|f\|_1.$$
(Note that in order for $\|f\|_1$ to be finite, $f$ must be countably supported, and so the inequalities just displayed follow from those for finitely supported $f$ by monotone convergence.) 

\medskip
\noindent
Indeed, in this discussion, there is nothing to prevent us from taking $\mathcal{L}$ and $\mathcal{L}_j$ to be the families of {\em all} lines in $\mathbb{F}^d$. Corresponding to the case of joints, we obtain:
\begin{theorem}\label{thm:joints}
Let $\mathbb{F}$ be an arbitrary field and let $\mathcal{L}^\ast$ be the family of all lines in $\mathbb{F}^d$. For every finitely supported $M:\mathbb{F}^d \to \mathbb{R}_+$ there is a nonnegative function $g(x, l)$ defined on $\mathbb{F}^d \times \mathcal{L}^\ast$ such that for all $x \in \mathbb{F}^d$ and $l \in \mathcal{L}^\ast$ we have
$$M(x)^d \delta(x, l_1, \dots , l_d) \leq g(x, l_1) \dots g(x, l_d)$$
and for all $l \in \mathcal{L}^\ast$,
$$ \sum_{x \in l} g(x, l) \lesssim \left( \sum_{x \in \mathbb{F}^d}M(x)^d\right)^{1/d}$$
where the implicit constant depends only on $d$.
\end{theorem}
\noindent
The function $g$ here is the kernel of the operator $R$ of Theorem~\ref{thm:main_sym}. We shall deduce this from the discrete and finite version of Theorem~\ref{thm:main_sym} in Section~\ref{sect:joints_pf}, having verified the auxiliary hypothesis via Theorem~\ref{thm:essence} in Section~\ref{sect:verif}. Theorem~\ref{thm:essence} may perhaps be of independent interest.

\medskip
\noindent
As a direct consequence of this result we have:
\begin{corollary}\label{thm:mjoints}
Let $\mathcal{L}_1, \dots, \mathcal{L}_d$ be finite families of lines in $\mathbb{F}^d$ and let $J$ be the set of their multijoints. Then for every $M:J \to \mathbb{R}_+$ there are nonnegative functions $g_j(x, l_j)$ such that for all $x \in J$ and $l_j \in \mathcal{L}_j$ we have
$$M(x)^d \delta(x, l_1, \dots , l_d) \leq g_1(x, l_1) \dots g_d(x, l_d)$$
and for all $j$, for all $l_j \in \mathcal{L}_j$,
$$ \sum_{x \in l_j} g_j(x, l_j) \lesssim \left( \sum_{x \in J}M(x)^d\right)^{1/d}$$
where the implicit constant depends only on $d$.
\end{corollary}
\medskip
\noindent
Note the similarity between these results and some of the estimates around the concept of visibility which were developed and used in Euclidean space by Guth \cite{Guth} and Bourgain and Guth \cite{MR2860188} in their analysis of the endpoint multilinear Kakeya inequality.

\medskip
\noindent
We also note that Theorem~\ref{thm:joints} has also been obtained in a more general form (using more direct and entirely different methods) by the second author, see \cite{Tang}. See also Theorem~\ref{thm:joints_upgrade} in Section~\ref{sect:concl} below where we indicate how the finite support hypothesis on $M$ may be dispensed with.

\subsection*{Acknowledgements} The first author 
was partially supported by Grant CEX2019-000904-S funded by MCIN/AEI/ 10.13039/501100011033 while visiting ICMAT in Madrid. The second author was supported by The Maxwell Institute Graduate School in Analysis and its
Applications (MIGSAA), a Centre for Doctoral Training funded by the UK Engineering and
Physical Sciences Research Council (grant EP/L016508/01), the Scottish Funding Council,
The University of Edinburgh and Heriot--Watt University.
The first author would like to thank Timo S. H\"anninen for various discussions relating to the material presented here.

\section{Proofs of Theorems~\ref{thm:main} and ~\ref{thm:main_sym} in the finite discrete setting}\label{sect:two}

\medskip
\noindent We first observe that it suffices to prove the desired conclusions for $q=1$, since the general case for $q > 1$ and $M \in L^{q'}$ follows from the case $q=1$ applied with the measure $M(x) {\rm d}\mu(x)$ in place of ${\rm d} \mu(x)$. Moreover, when $q =1$ it suffices to take $M \equiv 1$. 

\medskip
\noindent
We begin with Theorem~\ref{thm:main_sym} since the argument is a little simpler.

\medskip
\noindent
We will write the argument additively on finite discrete measure spaces, and in particular $\mathcal{Y}$ is a normed lattice over a finite set $Y$. The $x$-sums in what follows are all with respect to the weight $\mu$ on $X$ whose explicit appearance we suppress. Denote the kernel of $T$ by $K(x,y_1, \dots , y_d)$. The operator $R$ of the conclusion is determined by its kernel which we denote by $g(x,y)$. Let 
$$\mathcal{C} = \left\{ g(x, y) \geq 0 : K(x,y_1, \dots , y_d) \leq g(x, y_1) \dots g(x,y_d) \right\}$$
be the set of those $g$ satisfying our set of constraints (corresponding to the first conclusion of Theorem~\ref{thm:main_sym}). Note that $\mathcal{C}$ is nonempty and indeed there exist $g \in \mathcal{C}$ satisfying all of its defining inequalities with strict inequality. Moreover $\mathcal{C}$ is convex by the arithmetic-geometric mean inequality. Let
$\mathcal{F} = \{ f: Y \to \mathbb{R}_+ \, : \, \|f\|_\mathcal{Y} \leq 1\}$, and note that $\mathcal{F}$ is also convex. Corresponding to the second conclusion of Theorem~\ref{thm:main_sym}, we seek a $g \in \mathcal{C}$ such that
$$ \sup_{f \in \mathcal{F}}\sum_{x,y} g(x,y) f(y)\leq AB  ;$$
that is, we need to show that 
$$ \min_{g \in \mathcal{C}} \sup_{f \in \mathcal{F}} \; \sum_{y}\sum_{x} f(y) g(x,y) \leq AB.$$

\medskip
\noindent
Note that the mappings
$$ f \mapsto \sum_{y}\sum_{x} f(y) g(x,y) \mbox{ and } g \mapsto \sum_{y}\sum_{x} f(y) g(x,y) $$
for $g \in \mathcal{C}$ fixed and $f \in \mathcal{F}$ fixed respectively are linear, hence concave and convex. We may apply a minimax theorem\footnote{In the present context, we observe that Slater's condition is satisfied, so we may apply the results in \cite[p.226]{MR2061575}. Alternatively, we could use the more sophisticated lopsided minimax theorem found in \cite{MR749753} (and which was deployed in \cite{CHV1}), and which we shall need anyway in the general case. The topological hypotheses of that theorem reduce here to the continuity (for each fixed $f$) of the map $g \mapsto \sum_{x,y} f(y) g(x,y)$ and the existence of an $f \in \mathcal{F}$ (a suitable constant $f$ will work) such that the sets $\{g \in \mathcal{C} \, : \, 
\sum_{x,y} f(y) g(x,y) \leq \lambda\}$ are compact for all sufficiently large $\lambda$; with the usual topology on the underlying finite-dimensional Euclidean space these conditions are easily verified.} to obtain
$$ \min_{g \in \mathcal{C}} \sup_{f \in \mathcal{F}} \; \sum_{y}\sum_{x} f(y) g(x,y) = \sup_{f \in \mathcal{F}} \inf_{g \in \mathcal{C}} \; \sum_{y}\sum_{x} f(y) g(x,y).$$

\medskip
\noindent
Fix $f \in \mathcal{F}$. We now wish to show that 
\begin{equation*}\label{one}
\inf_{g \in \mathcal{C}} \; \sum_{x} \sum_{y} f(y) g(x,y) \leq 
AB.
\end{equation*}

\medskip
\noindent
The constraints given by $\mathcal{C}$ on different $x \in X$ are independent of each other, and so, with 
$$\mathcal{C}_x =\{ S(x, \cdot)
\, : \, K(x, y_1, \dots , y_d) \leq S(x,y_1) \dots S(x,y_d) \},$$ 
we have
$$\inf_{g \in \mathcal{C}} \; \sum_{x} \sum_{y} f(y) g(x,y)
=  \sum_{x} \inf_{S \in \mathcal{C}_x} \; \sum_{y} f(y) S(x,y).$$

\medskip
\noindent
By the auxiliary structural hypothesis, for each fixed $x$ and $f \in \mathcal{F}$, we have 
$$ \min_{S \in \mathcal{C}_x}\left(\sum_y f(y) S(x, y)\right)^d \leq B^d \sum_{y_1, \dots , y_d} K(x, y_1, \dots , y_d) f(y_1) \dots f(y_d),$$ 
and thus
$$\sum_{x}\inf_{S \in \mathcal{C}_x} \; \sum_{y} f(y) S(x,y)$$
$$ \leq
B\sum_{x}\left(\sum_{y_1, \dots , y_d}K(x, y_1, \dots , y_d) f(y_1) \dots f(y_d)\right)^{1/d}
\leq AB  $$
also by hypothesis. This concludes the argument for Theorem~\ref{thm:main_sym}.

\medskip
\noindent
Now we turn to Theorem~\ref{thm:main}. We recall that we are in the case $q=1$, and we continue in the finite and discrete setting. We use the shorthand notation $y = (y_1, \dots , y_d)$ where $y_j \in Y_j$ and $\mathcal{Y}_j$ is a normed lattice over the finite set $Y_j$. We are now looking for functions $g_1(x,y_1), \dots , g_d(x,y_d)$ satisfying the constraints defined by
  $$ \mathcal{C} = \{ (g_1, \dots g_d) \geq 0 \, : K(x,y)\leq g_1(x,y_1) \cdots g_d(x,y_d)\}$$
  and such that
  $$\sup_j  \sup_{\|f_j\|_{\mathcal{Y}_j} \leq 1}\sum_x \sum_{y_j \in Y_j} g_j(x, y_j)f_j(y_j)  \leq AB.$$
  Note that the set $\mathcal{C}$ is nonempty and convex (and once again there exist members of $\mathcal{C}$ satisfying the defining constraints with strict inequality). The left-hand side of the last expression can also be written as 
  $$ \sup_{\sum_j b_j = 1}  \sum_j b_j \sup_{\|f_j\|_{\mathcal{Y}_j} \leq 1} \sum_{x,y_j} g_j(x, y_j)f_j(y_j)$$
  or, equivalently, with $f_j$ replacing $b_j f_j$,
   $$ \sup_{\sum_j\|f_j\|_{\mathcal{Y}_j} \leq 1} \sum_j \sum_{x,y_j}g_j(x, y_j)f_j(y_j).$$
 Let
 $$\mathcal{F} = \{(f_1, \dots, f_d) \in (\mathcal{Y}_1 \times \dots \times \mathcal{Y}_d)_+ \, : \, \sum_j\|f_j\|_{\mathcal{Y}_j} \leq 1\},$$
 and note that $\mathcal{F}$ is also convex. What we are trying to show, then, is that
  $$ \min_{(g_1, \dots , g_d) \in \mathcal{C}} \sup_{(f_1, \dots , f_d) \in \mathcal{F}} \sum_j \sum_{x,y_j}g_j(x, y_j)f_j(y_j) \leq AB. $$
  
\medskip
\noindent
Once again we can use a minimax theorem to interchange the $\inf$ and the $\sup$, and therefore we wish to show
  $$\sup_{(f_1, \dots , f_d) \in \mathcal{F}}  \inf_{(g_1, \dots , g_d) \in \mathcal{C}}\sum_j \sum_{x,y_j} g_j(x, y_j)f_j(y_j) \leq AB.$$
  So fix $(f_j) \in \mathcal{F}$ and look at
  $$  \inf_{(g_1, \dots , g_d) \in \mathcal{C}} \sum_x \sum_j \sum_{y_j} g_j(x, y_j)f_j(y_j) = \sum_x \left(\inf_{(g_1, \dots , g_d) \in \mathcal{C}}\sum_j \sum_{y_j} g_j(x, y_j)f_j(y_j) \right)$$
  which is valid since the constraints imposed by $\mathcal{C}$ for different $x$ are independent of each other.
  
\medskip
\noindent
Temporarily fix $x$. So, with $K_x(y) = K(x,y)$, $S_j : Y_j \to \mathbb{R}_+$ let 
$$\mathcal{C}_x = \{(S_1, \dots , S_d) \, : K_x(y)\leq S_1(y_1) \cdots S_d(y_d)\}.$$
We are now looking at
$$ \inf_{(S_1, \dots , S_d) \in \mathcal{C}_x} \sum_j \sum_{y_j} S_j(y_j)f_j(y_j) = \inf_{(S_1, \dots , S_d) \in \mathcal{C}_x} \inf_{t_1 \dots t_d = 1} \frac{1}{d}\sum_j t_j d \sum_{y_j} S_j(y_j) f_j(y_j)$$
$$ =  \inf_{(S_1, \dots , S_d) \in \mathcal{C}_x}\prod_{j=1}^d
\left( d \sum_{y_j} S_j(y_j)f_j(y_j) \right)^{1/d}$$
$$ = d \inf_{(S_1, \dots , S_d) \in \mathcal{C}_x}\left( 
\sum_{y_1, \dots , y_d} S_1(y_1) \dots S_d(y_d) f_1(y_1) \dots f_d(y_d) \right)^{1/d}$$
$$ \leq Bd \left(\sum_{y} K(x,y) f_1(y_1) \dots f_d(y_d)\right)^{1/d}$$
where the first equality holds because $\mathcal{C}_x$ is invariant under replacing $S_j$ by $t_j S_j$ with $\prod_j t_j = 1$, the second equality holds by the arithmetic-geometric mean inequality and knowledge of its extremisers, and the final inequality holds because of the auxiliary hypothesis.

\medskip
\noindent
Therefore we have, using the main hypothesis,
$$\inf_{(g_1, \dots , g_d) \in \mathcal{C}}\sum_j \sum_{x, y_j} g_j(x, y_j)f_j(y_j)\leq Bd \sum_x \left(\sum_{y} K(x,y) f_1(y_1) \dots f_d(y_d) \right)^{1/d}$$
$$ = Bd  \| T(f_1, \dots , f_d)^{1/d}\|_1
\leq AB  d\prod_j \|f_j\|_{\mathcal{Y}_j}^{1/d}  \leq AB \sum_j \|f_j\|_{\mathcal{Y}_j} \leq  AB$$
as required, once again using the arithmetic-geometric mean inequality. This concludes the argument for Theorem~\ref{thm:main}.

\medskip
\noindent
{\bf Remark.} In the general case, the definition of the set $\mathcal{F}$ appearing in the proof will be amended to refer to the dense subspace of $\mathcal{Y}$ or $ \mathcal{Y}_1 \times \dots \times \mathcal{Y}_d$ which features in the auxiliary hypothesis. In the case of the application to joints, $\mathcal{F}$ will consist of finitely-supported functions defined on the class of all lines in $\mathbb{F}^d$. It is in the application of a suitable minimax theorem that we shall be required to provide substantial additional arguments relating to compactness in order to establish the full versions of Theorems~\ref{thm:main} and \ref{thm:main_sym}. 

\section{Verification of the auxiliary hypothesis for the multijoints kernel}\label{sect:verif}

\noindent
In this section we verify that the auxiliary structural hypotheses discussed above are indeed verified in the setting of the multijoints kernel $\delta$. Since $\delta$ is symmetric, it suffices to establish the symmetric version of the auxiliary hypothesis. 

\medskip
\noindent
Let $\mathcal{L}^\ast$ be the set of all lines in $\mathbb{F}^d$ and consider nonnegative $f$ belonging to the dense linear subspace of $l^1(\mathcal{L}^\ast)$ consisting of finitely supported functions. Focusing on the kernel of the desired operator $S$, we seek nonnegative $S(x,l)$ (defined on $\mathbb{F}^d \times \mathcal{L}^\ast$ and dependent on $f$) such that 
$$ S(x,l_1) \dots S(x,l_d) \geq 1 \mbox{ when } \delta(x,l_1, \dots, l_d) = 1$$
and 
$$ \sum_{l \in \mathcal{L}^\ast} S(x,l) f(l) \lesssim \left(\sum_{l_1, \dots ,l_d \in \mathcal{L}^\ast} \delta(x,l_1, \dots , l_d) f(l_1) \dots f(l_d)\right)^{1/d}.$$
This a pointwise task in $x$, and therefore it suffices to carry it out when $x=0$. 

\medskip
\noindent
Let $\mathcal{L}_0$ denote the set of all lines in $\mathbb{F}^d$ passing through $0$, and let
$\delta(l_1, \dots, l_d) = \delta(0,l_1, \dots, l_d)$. 
The following result thus verifies the auxiliary hypothesis for the case of joints, and is perhaps of independent interest. 

\begin{theorem}\label{thm:essence}
For each finitely supported nonnegative $f$ defined on $\mathcal{L}_0$, there is a function $S: \mathcal{L}_0 \to \mathbb{R}_+$ such that
$$ S(l_1) \dots S(l_d) \geq 1 \mbox{ when } \delta(l_1, \dots, l_d) = 1$$
and 
$$ \sum_{l \in \mathcal{L}_0} S(l) f(l) \lesssim \left(\sum_{l_1, \dots ,l_d \in \mathcal{L}_0} \delta(l_1, \dots , l_d) f(l_1) \dots f(l_d)\right)^{1/d}$$
where the implicit constant depends only on $d$.
\end{theorem}

\subsection{Preliminary remarks to the proof}
Note that, for $f$ fixed, we only need to define $S$ on those $l \in \mathcal{L}_0$ which are in the support of $f$. Fix $f$ nonnegative finitely supported on $\mathcal{L}_0$. We want to show that there is a set of weights $(S(l))_{l\in {\rm supp} \, f}$ satisfying 
$$ S(l_1) \dots S(l_d) \geq 1 {\mbox { whenever }} \delta(l_1, \dots, l_d) =1$$
(admissibility), and such that 
\[\left(\sum_{l \in {\rm supp} f}S(l)f(l)\right)^d \lesssim \sum_{l_1, \dots , l_d}\delta(l_1, \dots , l_d)f(l_1) \cdots f(l_d)\]
(main estimate).

\medskip
\noindent
The naive approach is to try the ansatz $S \equiv 1$ and see what happens in the main estimate.
Immediately we see there is an obstruction to it holding: if $f$ happens to be supported on a collection of lines which all lie in some fixed hyperplane, the right-hand side will be zero while the left-hand side need not be. Therefore our construction of $S$ will need to take account of this obstruction. When the obstruction is in place, however, we can simply set $S \equiv 0$ since the admissibility condition never comes into play. 

\medskip
\noindent
More generally, we may expect similar issues to present themselves 
when the total mass of $f$ is concentrated on some hyperplane $\pi$ in the sense that 
$$ \sum_{l \subseteq \pi} f(l) \gtrsim  \sum_{l \nsubseteq \pi} f(l).$$
Moreover, concentration of mass on a hyperplane may arise because of {\em a priori} concentration on a lower-dimensional subspace. It is therefore pertinent to consider in turn the possibility of concentration of mass of $f$ on lines, then on $2$-planes, then on $3$-planes, and so on up to and including hyperplanes. If there is no concentration occuring at any stage, we may hope to be able to take $S \equiv 1$; if concentration does occur at some stage, we may expect to take $S$ to be small on the lines contributing to the concentration and large on the remaining lines, in such a way that the admissibility condition holds. If concentration occurs at multiple stages, we will correspondingly take a graded approach to defining $S$.

\subsection{Construction of an increasing sequence of subspaces and its immediate properties} Let $1 =\alpha_1 < \alpha_2 < \dots < \alpha_{d-1}$ which we take to be given by $\alpha_k = 2^{k-1}$.
We first look for the smallest $k_1$, with $1\leq k_1 \leq d-1$, for which there is an {\em $\alpha_{k_1}$-heavy $k_1$-plane} $\pi_1$, i.e. for which  
$$ \sum_{l \subseteq  \pi_1} f(l) > \alpha_{k_1} \sum_{l \not\subseteq  \pi_1} f(l).$$
Of course there may be no such $\pi_1$. If there is one, we next look for the smallest $k_2$, with $k_1 < k_2 \leq d-1$, for which there is an {\em $\alpha_{k_2}$-heavy $k_2$-plane $\pi_2$ which contains $\pi_1$}, i.e. for which
$$ \sum_{l \subseteq  \pi_2, l \not\subseteq \pi_1} f(l) > \alpha_{k_2} \sum_{l \not\subseteq  \pi_2} f(l).$$
Again, there may be no such $\pi_2$. If there is one, we next look for the smallest $k_3$, with $k_2 < k_3 \leq d-1$, for which there is an {\em $\alpha_{k_3}$-heavy $k_3$-plane $\pi_3$ which contains $\pi_2$}, i.e. for which
$$ \sum_{l \subseteq  \pi_3, l \not\subseteq \pi_2} f(l) > \alpha_{k_3} \sum_{l \not\subseteq  \pi_3} f(l).$$
We continue this process until we are forced to stop -- either because we never started if there were no $\alpha_k$-heavy planes of any dimension $k$ less than $d$, or because we have arrived at some $\pi_N$ of dimension $k_N = d-1$, or because we have a $\pi_N$ of dimension $k_N < d-1$ but there are no $\alpha_{k}$-heavy $k$-planes for any $k_N < k \leq d-1$.

\medskip
\noindent
We shall deal separately with the case in which the sequence of subspaces is empty, and so we assume we have a non-empty maximal\footnote{in the sense that it cannot be extended} increasing sequence  
of subspaces $\pi_1 \subsetneq \pi_2 \subsetneq \dots \subsetneq \pi_N$ (with $N \geq 1$) of dimensions $(1 \leq)\, k_1 < k_2 < \dots < k_N \,(\leq d-1)$ respectively such that for each $1 \leq n \leq N$,\footnote{We interpret the condition $l \not\subseteq \pi_0$ to be the void condition.} 
\begin{equation}\label{eq:heavy}
\sum_{l \subseteq  \pi_{n}, l \not\subseteq \pi_{n-1}} f(l) > \alpha_{k_n} \sum_{l \not\subseteq  \pi_{n}} f(l)
\end{equation}
and such that for all subspaces $\pi$ which contain $\pi_{n-1}$ and which satisfy $ k_{n-1} < {\rm dim} \, \pi < k_{n}$ we have
\begin{equation*}
\sum_{l \subseteq  \pi, l \not\subseteq \pi_{n-1}} f(l) \leq \alpha_{{\rm dim} \, \pi } \sum_{l \not\subseteq  \pi} f(l) 
\end{equation*}
and thus
\begin{equation}\label{eq:light}
\sum_{l \subseteq  \pi, l \not\subseteq \pi_{n-1}} f(l) \leq \alpha_{k_n -1} \sum_{l \not\subseteq  \pi} f(l).
\end{equation}

\medskip
\noindent
We can qualitatively improve the right-hand side of this inequality to include the additional constraint that $l \subseteq \pi_n$ at the expense of a multiplicative constant: 

\begin{lemma}\label{lemma:improved_lightness}
We have
\begin{equation}\label{eq:light_c} 
\sum_{l \subseteq  \pi, l \not\subseteq \pi_{n-1}} f(l) \leq 4 \alpha_{k_n -1}\sum_{l \subseteq  \pi_{n}, l \not\subseteq \pi} f(l).
\end{equation}
\end{lemma}

\begin{proof}
Notice that for each $n$ and for each subspace $\pi$ which contains $\pi_{n-1}$ and which is strictly contained in $\pi_n$ we have, by \eqref{eq:light} and \eqref{eq:heavy},
$$\sum_{l \subseteq  \pi, l \not\subseteq \pi_{n-1}} f(l) \leq \alpha_{k_n -1} \sum_{l \not\subseteq  \pi} f(l) = \alpha_{k_n -1} \sum_{l \not\subseteq  \pi, l \subseteq \pi_n} f(l) + \alpha_{k_n -1} \sum_{l \not\subseteq  \pi_n} f(l)$$
$$ < \alpha_{k_n -1} \sum_{l \not\subseteq  \pi, l \subseteq \pi_n} f(l) + \frac{\alpha_{k_n -1}}{\alpha_{k_n}}\sum_{l \subseteq  \pi_{n}, l \not\subseteq \pi_{n-1}} f(l)$$
$$ = \alpha_{k_n -1} \sum_{l \not\subseteq  \pi, l \subseteq \pi_n} f(l) 
+ \frac{\alpha_{k_n -1}}{\alpha_{k_n}}\sum_{l \subseteq  \pi_{n}, l \not\subseteq \pi} f(l) 
+ \frac{\alpha_{k_n -1}}{\alpha_{k_n}}\sum_{l \subseteq  \pi, l \not\subseteq \pi_{n-1}} f(l).$$
Rearranging this inequality gives
$$ \left(1- \frac{\alpha_{k_n -1}}{\alpha_{k_n}}\right) \sum_{l \subseteq  \pi, l \not\subseteq \pi_{n-1}} f(l) \leq \alpha_{k_n -1}\left(1 + \frac{1}{\alpha_{k_n}}\right)\sum_{l \subseteq  \pi_{n}, l \not\subseteq \pi} f(l) $$
or
\begin{equation*}
\sum_{l \subseteq  \pi, l \not\subseteq \pi_{n-1}} f(l) \leq \alpha_{k_n -1} \left(\frac{\alpha_{k_n} +1}{\alpha_{k_n} - \alpha_{k_{n-1}}}\right) \sum_{l \subseteq  \pi_{n}, l \not\subseteq \pi} f(l).
\end{equation*}
With $\alpha_k = 2^{k-1}$ we have $1 < (\alpha_{k_n} +1)/(\alpha_{k_n} - \alpha_{k_{n-1}}) \leq 4$, and we are done.
\end{proof} 

\subsection{Arranging the mass of $f$ into layers}
For $1 \leq n \leq N+1$ let 
$$ F_n = \sum_{l \subseteq \pi_n, l \not\subseteq \pi_{n-1}} f(l)$$
(where we interpret the conditions $l \not\subseteq \pi_0$ and  $l \subseteq \pi_{N+1}$ as void). 
Note that by \eqref{eq:heavy}, we immediately have
\begin{equation}\label{eq:F_q_monotone}
F_n > \alpha_{k_n} F_{n+1}.
\end{equation}

\medskip
\noindent
For a $k$-plane $\pi$ and a $K$-plane $\Pi$ with $K >k$, and $l_1, \dots ,l_{K-k} \subseteq \Pi$, but $l_1, \dots ,l_{K-k} \not\subseteq \pi$, let $\delta_{\pi, \Pi}(l_1, \dots ,l_{K-k}) = 1$ if $e(l_1), \dots , e(l_{K-k})$ are linearly independent and $\delta_{\pi, \Pi}(l_1, \dots, l_{K-k}) = 0$ otherwise. We have the following lemma which is proved below in Section~\ref{sect:L2pf}:
\begin{lemma}\label{lemma:concl} For some constants $\beta_n$, we have for $1 \leq n \leq N+1$,
$$ F_n^{k_n - k_{n-1}} \leq \beta_n \sum_{\substack{ l_{k_{n-1} + 1}, \dots , l_{k_n} \subseteq \pi_n,\\ l_{k_{n-1} + 1}, \dots , l_{k_n} \not \subseteq \pi_{n-1}}}
\delta_{\pi_{n-1}, \pi_n}(l_{k_{n-1} + 1}, \dots , l_{k_n}) f(l_{k_{n-1} + 1}) \dots f(l_{k_n})$$
(where we take $k_0 = 0$ and $k_{N+1} = d$.)
\end{lemma}

\medskip
\noindent
In particular this lemma applies in the exceptional case that the sequence of  subspaces is empty, in which case the conclusion reads as
$$ \left(\sum_{l} f(l)\right)^d
\lesssim \sum_{l_1, \dots ,l_d} \delta(l_1, \dots , l_d) f(l_1) \dots f(l_d).$$
This demonstrates that the choice $S \equiv 1$ satisfies the main condition (as well as trivially the admissibility condition) in the exceptional case.

\medskip
\noindent
Note that we have as a direct consequence of the lemma that
\begin{equation}\label{eq:F_prod}
F_1^{k_1 - k_0} \cdots F_{n}^{k_n - k_{n-1}} \cdots F_{N+1}^{d - k_{N}} \leq \left(\prod_{n=1}^{N+1} \beta_n\right) \sum_{l_1, \dots , l_d}\delta(l_1, \dots , l_d)f(l_1) \cdots f(l_d).
\end{equation}

\subsection{Definition of $S$ and the main condition} 
We assume the sequence of subspaces is nonempty. For parameters $\rho_1, \dots ,\rho_{N+1}$ defined below, we define $S$ by
\[S(l) = 
\begin{cases} 
\rho_1 & l \subseteq \pi_1\\ 
\rho_2 & l \subseteq  \pi_2, l \not\subseteq \pi_1, \\
\vdots\\
\rho_n & l \subseteq  \pi_n, l \not\subseteq \pi_{n-1}, \\
\vdots\\
\rho_N & l \subseteq  \pi_N, l \not\subseteq \pi_{N-1} \\
\rho_{N+1} & l \not\subseteq \pi_{N}. \\
\end{cases}\]

\medskip
\noindent
Using the arrangement of the mass of $f$ into layers, we therefore have
$$ \sum_l S(l) f(l) = \rho_1 F_1 + \dots + \rho_n F_n + \dots + \rho_{N+1} F_{N+1}.$$ 

\medskip
\noindent
In the light of \eqref{eq:F_prod}, we shall therefore choose 
$$ \rho_n = F_n^{-1} \prod_{n=1}^{N+1}F_n^{\frac{k_n - k_{n-1}}{d}}$$ where $k_0=0$ and $k_{j+1} = d$, and this verifies the main condition on $S$.

\subsection{The admissibility condition}
Again we may assume that the sequence of subspaces is nonempty. If $\delta(l_1, \dots , l_d) =1$, the worst-case scenario is that there are $k_1$ $l$'s contained in $\pi_1$, $(k_2 - k_1)$ $l$'s contained in $\pi_2$ but not contained in $\pi_1$, etc. So the worst value of $S(l_1) \cdots S(l_d)$ will be 
\begin{align}\label{eq:worstS}
\rho_1^{k_1}\rho_2^{k_2 - k_{1}} \cdots \rho_{N+1}^{d-k_N}
= F_1^{-k_1}F_2^{-(k_2 - k_1)} \cdots 
F_{N+1}^{-(d-k_N)} \left(\prod_{n=1}^{N+1}F_n^{\frac{k_n - k_{n-1}}{d}} \right)^d=1.
\end{align}
With $m_n$ denoting the number of
$l$'s contained in $\pi_n$ but not contained in $\pi_{n-1}$, the general case follows from:
\begin{lemma}
Suppose that $m_1 + \dots + m_{N+1} = d$ and $m_n \leq k_n$ for all $1 \leq n \leq N+1$. Then
$$ \rho_1^{m_1} \cdots \rho_n^{m_n} \cdots  \rho_{N+1}^{m_{N+1}} \geq 1.$$
\end{lemma}
\begin{proof}
We write
$$\rho_1^{m_1} \cdots \rho_n^{m_n} \cdots  \rho_{N+1}^{m_{N+1}}$$
$$= 
\left(\rho_1^{k_1} \cdots \rho_n^{k_n - k_{n-1}} \cdots \rho_{N+1}^{d-k_N} \right) \left(\rho_1^{m_1 - k_1} \cdots \rho_n^{m_n- (k_n - k_{n-1})} \cdots  \rho_{N+1}^{m_{N+1} - (d - k_N)}\right)$$
$$ = \rho_1^{m_1 - k_1} \cdots \rho_n^{m_n- (k_n - k_{n-1})} \cdots  \rho_{N+1}^{m_{N+1} - (d - k_N)}$$
$$ = F_1^{k_1 - m_1} \cdots F_n^{k_n - k_{n-1} - m_n} \cdots 
F_{N+1}^{d-k_N -m_{N+1}}\geq F_{N+1}^0 =1 $$
by \eqref{eq:worstS} and repeated use of \eqref{eq:F_q_monotone} (noting that each $\alpha_n \geq 1$).
\end{proof}

\subsection{Proof of Lemma~\ref{lemma:concl}}\label{sect:L2pf}
It is here that we finally use the fact that for all subspaces $\pi$ which contain $\pi_{n-1}$ and which are strictly contained in $\pi_n$, inequality \eqref{eq:light_c} holds.

\medskip
\noindent
For ease of notation, fix $n$, let $k_{n-1} =k$ and let $k_n = k+r$. Relabel $\pi_{n-1}$ as $\Pi_0$ and $\pi_n$ as $\Pi_1$. Let $\Delta$ be the set of all lines which are contained in $\Pi_1$ but which are not contained in $\Pi_0$. 

\medskip
\noindent
Inequality \eqref{eq:light_c} of Lemma~\ref{lemma:improved_lightness} now becomes: for all subspaces $\pi$ with $\Pi_0 \subsetneq \pi \subsetneq \Pi_1$
\begin{equation}\label{eq:light_b}
\sum_{l \subseteq  \pi, l \not\subseteq \Pi_{0}} f(l) \leq 4 \alpha_{k_{n-1}} \sum_{l \not\subseteq  \pi, l \subseteq \Pi_1} f(l).
\end{equation}

\medskip
\noindent
We wish to prove
\begin{equation}\label{eq:to_prove}
\left(\sum_{l \in \Delta} f(l)\right)^{r} \leq \beta_n \sum_{l_{1}, \dots , l_{r} \in \Delta}
\delta_{\Pi_{0}, \Pi_1}(l_{1}, \dots , l_{r}) f(l_{1}) \dots f(l_{r}).
\end{equation}

\medskip
\noindent
Let $(l_1, \dots , l_r)$ be an $r$-tuple of lines, each of which lies in $\Delta$. Then, together with $\Pi_0$, they span a subspace of $\Pi_1$ of dimension $j$ for some $j \in \{k+1, \dots, k+r\}$. For $k+1 \leq j \leq k+r$ let
$$\Gamma_j = \{(l_1, \dots , l_r) \in \Delta^r \, : \, {\rm dim \; span} \{\Pi_0, l_1, \dots , l_r\} = j\}.$$
Note that $\delta_{\Pi_{0}, \Pi_1}(l_{1}, \dots , l_{r}) = 1$ if and only if $(l_1, \dots, l_r) \in \Gamma_{k+r}$.

\medskip
\noindent
We expand the left-hand side of \eqref{eq:to_prove} 
as
$$
\sum_{j=k+1}^{k+r} \sum_{(l_1, \dots , l_r) \in \Gamma_j} f(l_1) \dots f(l_r),
$$
and to prove \eqref{eq:to_prove} it suffices (indeed it is equivalent) to show that for $k+1 \leq j < k+r$,
$$ \sum_{(l_1, \dots , l_r) \in \Gamma_j} f(l_1) \dots f(l_r) \lesssim \sum_{(l_1, \dots , l_r) \in \Gamma_{k+r}} f(l_1) \dots f(l_r), 
$$
and this in turn follows if we can show that for
$k+1 \leq j <k+r$
$$ \sum_{(l_1, \dots , l_r) \in \Gamma_j} f(l_1) \dots f(l_r) \lesssim \sum_{(l_1, \dots , l_r) \in \Gamma_{j+1}} f(l_1) \dots f(l_r). 
$$

\medskip
\noindent
If $(l_1, \dots ,l_r) \in \Gamma_j$, then for some $(j-k)$-tuple -- which is without loss of generality $(l_1, \dots , l_{j-k})$ -- we have that $\{\Pi_0, l_1, \dots , l_{j-k}\}$ spans a $j$-plane $H_{\{l_1, \dots , l_{j-k}\}}$ satisfying $\Pi_0 \subsetneq H_{l_1, \dots , l_{j-k}} \subsetneq \Pi_1$.
Therefore
$$ \sum_{(l_1, \dots , l_r) \in \Gamma_j} f(l_1) \dots f(l_r)$$
$$\leq {r \choose j-k} \sum_{\{e(l_1), \dots e(l_{j-k})\} 
{\rm{ lin. \, indep.}}} f(l_1) \dots f(l_{j-k})
\sum_{l_{j-k+1}, \dots , l_r \subseteq H_{\{l_1, \dots , l_{j-k}\}}} f(l_{j-k+1}) \dots f(l_r)$$
$$ = {r \choose j-k} \sum_{\{e(l_1), \dots e(l_{j-k})\} 
{\rm{ lin. \, indep.}}} f(l_1) \dots f(l_{j-k})
\left(\sum_{l \subseteq H_{\{l_1, \dots , l_{j-k}\}}} f(l)\right)^{r-(j-k)}.$$
We use \eqref{eq:light_b} -- which is applicable since $\Pi_0 \subsetneq H_{l_1, \dots , l_{j-k}} \subsetneq \Pi_1$ -- to estimate the bracketed expression here by
$$4 \alpha_{k} \left(\sum_{l \subseteq H_{\{l_1, \dots , l_{j-k}\}}} f(l)\right)^{r-(j-k)-1}\left(\sum_{l \not\subseteq H_{\{l_1, \dots , l_{j-k}\}}, \, l \subseteq \Pi_1} f(l)\right).$$
This shows that 
$$ \sum_{(l_1, \dots , l_r) \in \Gamma_j} f(l_1) \dots f(l_r) \lesssim \sum_{(l_1, \dots , l_r) \in \Gamma_{j+1}} f(l_1) \dots f(l_r),
$$
as needed.

\medskip
\noindent
This completes the proof of Lemma~\ref{lemma:concl} and Theorem~\ref{thm:essence}.

\section{Proof of Theorem~\ref{thm:joints}}\label{sect:joints_pf}
\noindent
In the case that the field $\mathbb{F}$ is finite we can simply let $X= \mathbb{F}^d$ with counting measure, $Y = \mathcal{L}^\ast$ (the set of all lines in $\mathbb{F}^d$) with counting measure, $\mathcal{Y} = l^1(Y)$ and let $T: \mathcal{Y}^d \to \mathcal{M}(X)$ be given as above by
$$ T(f_1, \dots , f_d)(x) 
= \sum_{l_j \in \mathcal{L}^\ast} \delta(x,l_1, \dots , l_d) f_1(l_1) \dots f_d(l_d).$$
Then $T$ saturates $X$, and by Zhang's theorem we have
$$\|T(f, \dots , f)^{1/d}\|_{L^{d/(d-1)}(X)} \lesssim \|f\|_{L^1(Y)}.$$
We have shown in Section~\ref{sect:verif} that $T$ satisfies the auxiliary structural hypothesis given by \eqref{eq:seven_sym} and \eqref{eq:eight_sym}. Thus, by the discrete and finite version of Theorem~\ref{thm:main_sym}, there exists $g(x,l) \geq 0$ (the kernel of the operator $R$) defined on $X \times \mathcal{L}^\ast$ such that 
\begin{equation*}
M(x)^d\delta(x, l_1, \dots , l_d) \leq g(x, l_1) \dots g(x, l_d)
\end{equation*}
and
\begin{equation*}
\sum_{x \in l} g(x,l) \lesssim \left( \sum_{x \in \mathbb{F}^d}M(x)^d\right)^{1/d}
\end{equation*}
uniformly in $l \in \mathcal{L}^\ast$, where the implicit constant depends only on $d$. This completes the argument when $\mathbb{F}$ is finite.

\medskip
\noindent
For the general case when the field is infinite, we will first need a straightforward linear-algebraic lemma:
\begin{lemma}
Let $V$ be a finite-dimensional vector space over a field $\mathbb{F}$. Let $A \subseteq V$ be a finite subset of $V$. Then there is a finite set $B$ satisfying $A \subseteq B \subseteq V$ such that for every subset $C \subseteq B$ of linearly independent vectors, there is a subset $D \subseteq B \setminus C$ such that $C \cup D$ forms a basis for $V$. 
\end{lemma}

\begin{proof}
Let $U = {\rm span}\, A$ and let $F$ be a linearly independent set of vectors in $V \setminus U$ such that ${\rm span} \, (A \cup F) =V$. Let $B = A \cup F$. If $C \subseteq B$ is a linearly independent set of vectors, then $C \cap A$ can be extended to a basis of $U$ by adding vectors from $A \setminus C$, and $C \cap F$ can be extended to a basis of ${\rm span} \, F$ by adding vectors from $F$. Thus $C$ can be extended to a basis of $V$ using vectors from $B$.  
\end{proof}

\medskip
\noindent
In order to simplify notation we assume from now on that $\left( \sum_{x \in \mathbb{F}^d}M(x)^d\right)^{1/d} = \|M\|_d=1$.

\medskip
\noindent
It suffices to define $g: {\rm supp} \, M \times \mathcal{L}^\ast \to \mathbb{R}_+$, such that for all $x \in {\rm supp} \, M$, whenever $l_1, \dots , l_d \in \mathcal{L}^\ast$ form a joint at $x$ we have 
\begin{equation*}
M(x)^d \leq g(x, l_1) \dots g(x, l_d)
\end{equation*}
and
\begin{equation*}
\sum_{x \in l} g(x,l) \lesssim 1
\end{equation*}
uniformly in $l \in \mathcal{L}^\ast$. We will define $g$ in a piecemeal fashion: we first identify a finite family $\mathcal{L}$ of lines, and define $g$ on ${\rm supp} \, M \times \mathcal{L}$; and then we define $g$ on ${\rm supp} \, M \times (\mathcal{L}^\ast \setminus \mathcal{L})$. Finally we check that our $g$ satisfies the desired conclusions.

\medskip
\noindent
We identify the finite family of lines $\mathcal{L}$. 
Included in $\mathcal{L}$ is the set $\mathcal{L}'$ of all lines containing two or more points of ${\rm supp} \, M$. For each point $x \in {\rm supp} \, M$ we consider the set $\mathcal{L}'_x$ of lines in $\mathcal{L}'$ which contain $x$. By the linear-algebraic lemma, there is a finite set $\mathcal{L}''_x$ of lines containing $x$, which contains $\mathcal{L}'_x$, such that every subset of $\mathcal{L}''_x$ whose members have linearly independent directions 
can be augmented by lines from $\mathcal{L}''_x$
to form a joint at $x$. We include all these sets $\mathcal{L}''_x$, for $x \in {\rm supp} \, M$, in our set $\mathcal{L}$. Clearly $\mathcal{L}$ is finite, and any line not in $\mathcal{L}$ contains at most one point of ${\rm supp} \, M$.

\medskip
\noindent
We now define $g$ on $J \times \mathcal{L}$ where $J$ is the set of joints of $\mathcal{L}$. This in particular defines $g$ on ${\rm supp} \, M \times \mathcal{L}$. Indeed, we simply apply the discrete and finite version of Theorem~\ref{thm:main_sym} to $\mathcal{L}$ and $J$ (observing that the saturation and auxiliary hypotheses hold) to obtain $g(x,l)$ for $x \in J$ and $l\in \mathcal{L}$ which is such that for all $x \in {\rm supp} \, M$, whenever $l_1, \dots , l_d \in \mathcal{L}$ form a joint at $x$ we have 
\begin{equation}\label{eq:zero}
M(x)^d \leq g(x, l_1) \dots g(x, l_d)
\end{equation}
and, for all $l \in \mathcal{L}$
\begin{equation}\label{eq:onex}
\sum_{x \in l \cap J} g(x,l) \leq C_d 
\end{equation}
Notice that this last inequality implies that for all $l \in \mathcal{L}$ and all $x_0 \in J$ we have $g(x_0,l) \leq C_d.$ 

\medskip
\noindent
Next we define $g$ on ${\rm supp} \, M \times (\mathcal{L}^\ast \setminus \mathcal{L})$. As we noted above, for any $l \in \mathcal{L}^\ast \setminus \mathcal{L}$, $l$ contains at most one point of ${\rm supp} \, M$. If $l$ contains no point of ${\rm supp} \, M$ we define $g(x,l) = 0$ for all $x \in {\rm supp} \, M$. If $l \cap {\rm supp} \, M = \{x_0\}$ we define 
$g(x_0,l) = C_d
$ and $g(x,l) = 0$ for all $x \in {\rm supp} \, M \setminus \{x_0\}$. Note that for all $l \in 
\mathcal{L}^\ast \setminus \mathcal{L}$, 
\begin{equation}\label{eq:twox}
\sum_{x \in l \cap \, {\rm supp} \, M}g(x,l) \leq C_d 
\end{equation}
since there can be at most one term in the sum on the left-hand side.

\medskip
\noindent
Finally we verify our desired conclusions. The second conclusion is immediate from \eqref{eq:onex} and \eqref{eq:twox}. For the first conclusion, take an $x_0 \in {\rm supp} \, M$ and 
let $l_1, \dots , l_d$ be a $d$-tuple of lines forming a joint at $x_0$, some of which may be in $\mathcal{L}$ and some of which may be in $\mathcal{L}^\ast \setminus \mathcal{L}$. If all the lines lie in $\mathcal{L}$, \eqref{eq:zero} directly gives what we need. 
Otherwise, for some $0 \leq k \leq d-1$, there are $k$ lines $l_1, \dots, l_k$ in $\mathcal{L}$ and $d-k$ lines $l_{k+1}, \dots ,l_d$ in $\mathcal{L}^\ast \setminus \mathcal{L}$. By our construction of $\mathcal{L}$ we may augment $  l_1, \dots, l_k$ with lines $l'_{k+1}, \dots ,l'_d$ in $\mathcal{L}$ to form a joint at $x_0$. And we have
$$ g(x_0, l_j) \geq g(x_0, l'_j) \mbox{ for } k + 1 \leq j \leq d$$
because each instance of the left-hand side is exactly $C_d$ 
while each instance of the right-hand side is at most $C_d$. Therefore,
$$ g(x_0, l_1) \dots g(x_0, l_d) \geq g(x_0, l_1) \dots g(x_0, l_k)g(x_0, l'_{k+1}) \dots g(x_0, l'_d)\geq M(x_0)^d$$
since $l_1, \dots , l_k, l'_{k+1}, \dots , l'_d$
form a joint at $x_0$. This completes the verification of the first conclusion.

\section{Concluding remarks}\label{sect:concl}
\noindent
The verification of the auxiliary hypothesis in the case of joints formed by higher-dimensional planes
appears to require more work, and will be addressed elsewhere. Nevertheless, once this is achieved, we will have an independent proof of results analogous to Theorem~\ref{thm:main} which have already been obtained directly by the second author in \cite{Tang}.

\medskip
\noindent
Verification of the auxiliary hypothesis in the context of the Euclidean multilinear Kakeya problem is another matter which we are currently addressing.

\medskip
\noindent
The details of the proofs of the full versions of Theorems~\ref{thm:main} and \ref{thm:main_sym} will also appear elsewhere. As remarked previously, it is not clear to what extent the auxiliary hypotheses are necessary for the conclusions of Theorems~\ref{thm:main} and \ref{thm:main_sym} to hold. 

\medskip
\noindent
Finally, we observe that we can remove the hypothesis in Theorem~\ref{thm:joints} that $M$ be finitely supported if, instead of its discrete and finite version, we use the full force of Theorem~\ref{thm:main_sym}:
\begin{theorem}\label{thm:joints_upgrade}
Let $\mathbb{F}$ be an arbitrary field and let $\mathcal{L}^\ast$ be the family of all lines in $\mathbb{F}^d$. For every $M:\mathbb{F}^d \to \mathbb{R}_+$ there is a nonnegative function $g(x, l)$ defined on $\mathbb{F}^d \times \mathcal{L}^\ast$ such that for all $x \in \mathbb{F}^d$ and $l \in \mathcal{L}^\ast$ we have
$$M(x)^d \delta(x, l_1, \dots , l_d) \leq g(x, l_1) \dots g(x, l_d)$$
and for all $l \in \mathcal{L}^\ast$,
$$ \sum_{x \in l} g(x, l) \lesssim \left( \sum_{x \in \mathbb{F}^d}M(x)^d\right)^{1/d}.$$
\end{theorem}
\begin{proof}
If $\sum_{x \in \mathbb{F}^d} M(x)^d = \infty$ there is nothing to prove. So we may assume 
that $\sum_{x \in \mathbb{F}^d} M(x)^d < \infty$ and that therefore $M$ is countably supported. Take $X$ to be the support of $M$ with counting measure. Let $Y = \mathcal{L}^\ast$ with counting measure, $\mathcal{Y} = l^1(Y)$ and let $T: \mathcal{Y}^d \to \mathcal{M}(X)$ be given as above by
$$ T(f_1, \dots , f_d)(x) 
= \sum_{l_j \in \mathcal{L}^\ast} \delta(x,l_1, \dots , l_d) f_1(l_1) \dots f_d(l_d).$$ 
Then $T$ saturates $X$, and by Zhang's theorem we have
$$\|T(f, \dots , f)^{1/d}\|_{L^{d/(d-1)}(X)} \lesssim \|f\|_{L^1(Y)}.$$
We note that $T$ satisfies the auxiliary structural hypothesis given by \eqref{eq:seven_sym} and \eqref{eq:eight_sym}, where the dense subspace of $l^1(\mathcal{L}^\ast)$ is taken to be the space of finitely supported functions defined on $\mathcal{L}^\ast$. Thus, by Theorem~\ref{thm:main_sym}, there exists $g(x,l) \geq 0$ defined on $X \times \mathcal{L}^\ast$ such that for all $x \in X$ and $l_1, \dots , l_d \in \mathcal{L}^\ast$,
\begin{equation*}
M(x)^d\delta(x, l_1, \dots , l_d) \leq g(x, l_1) \dots g(x, l_d)
\end{equation*}
and
\begin{equation*}
\sum_{x \in l} g(x,l) \lesssim \left( \sum_{x \in \mathbb{F}^d}M(x)^d\right)^{1/d}
\end{equation*}
uniformly in $l \in \mathcal{L}^\ast$. For $x \notin X$ and arbitrary $l \in \mathcal{L}^\ast$ we can take $g(x,l) = 0$.
\end{proof}

\bibliographystyle{plain}

\bibliography{bbb}

\end{document}